\documentclass[11pt]{article}
\usepackage{subfigure,epsfig,makeidx,amsmath,amsfonts,amssymb,latexsym,color,amsthm}
\newtheorem{theorem}{Theorem}[section]
\newtheorem{prop}{Proposition}[section]
\newtheorem{lemma}{Lemma}[section]

\newtheorem{defi}{Definition}[section]
\newtheorem{remark}{Remark}[section]
\newcommand{\E}{{\mathbb E}}
\newcommand{\RR}{{\mathbb R}}
\newcommand{\Z}{{\mathbb Z}}
\newcommand{\PP}{\mathbb P}
\newcommand{\poi}{\mathcal{P}}
\newcommand{\mat}{\mathcal{M}}

\newcounter{mycount}
\newenvironment{romlist}{\begin{list}{\rm(\roman{mycount})}%
   {\usecounter{mycount}\labelwidth=1cm\itemsep -0pt}}{\end{list}}

\newenvironment{letlist}{\begin{list}{{\rm(\alph{mycount})}}%
   {\usecounter{mycount}\labelwidth=1cm\itemsep 0pt}}{\end{list}}

\begin{document}

\title{Percolation in invariant Poisson graphs with\\ i.i.d.\ degrees}

\author{Maria Deijfen
\thanks{Department of Mathematics, Stockholm University, 106 91 Stockholm.  mia at math.su.se}
\and Olle H\"{a}ggstr\"{o}m
\thanks{Department of Mathematics, Chalmers University of Technology. olleh at chalmers.se}
\and Alexander E.\ Holroyd\thanks{Microsoft Research, 1
Microsoft Way, Redmond, WA 98052, USA; \& University of British
Columbia, 121-1984 Mathematics Rd., Vancouver, BC V6T 1Z2,
Canada. \mbox{holroyd~at~microsoft.com}}}

\date{8 February 2010}

\maketitle

\thispagestyle{empty}

\begin{abstract}

Let each point of a homogeneous Poisson process in $\RR^d$
independently be equipped with a random number of stubs
(half-edges) according to a given probability distribution
$\mu$ on the positive integers.  We consider
translation-invariant schemes for perfectly matching the stubs
to obtain a simple graph with degree distribution $\mu$.
Leaving aside degenerate cases, we prove that for any $\mu$
there exist schemes that give only finite components as well as
schemes that give infinite components.  For a particular
matching scheme that is a natural extension of Gale-Shapley
stable marriage, we give sufficient conditions on $\mu$ for the
absence and presence of infinite components.
\end{abstract}

\renewcommand{\thefootnote}{}

\section{Introduction}

Let $\poi$ be a homogeneous Poisson process with intensity 1 on
$\RR^d$. Furthermore, let $\mu$ be a probability measure on the
strictly positive integers.  We shall study
translation-invariant simple random graphs whose vertices are
the points of $\poi$ and where the degrees of the vertices are
i.i.d.\ with law $\mu$. Deijfen \cite{D} studied moment
properties achievable for the edge lengths in such graphs.
Here, we shall instead be interested in the
percolation-theoretic question of whether the graph contains a
component with infinitely many vertices.

We next formally describe the objects that we will work with.
For any random point measure $\Lambda$ we write
$[\Lambda]:=\{x\in\RR^d:\Lambda(\{x\})>0\}$ for its support, or
point-set. Let $\xi$ be a random integer-valued measure on
$\RR^d$ with the same support as $\poi$, and which, conditional
on $\poi$, assigns i.i.d.\ values with law $\mu$ to the
elements of $[\poi]$. The pair $(\poi,\xi)$ is a marked point
process with positive integer-valued marks.  For $x\in[\poi]$
we write $D_x$ for $\xi(\{x\})$, which we interpret as the
number of stubs at vertex $x$.

A {\bf matching scheme} for a marked process $(\poi,\xi)$ is a
point process $\mat$ on $(\RR^d)^2$ with the property that
almost surely for every pair $(x,y)\in[\mat]$ we have
$x,y\in[\poi]$, and such that in the graph $G=G(\poi,\mat)$
with vertex set $[\poi]$ and edge set $[\mat]$, each vertex $x$
has degree $D_x$.  The matching schemes under consideration
will always be {\bf simple}, meaning that $G$ has no self-loops
and no multiple edges, and {\bf translation-invariant}, meaning
that $\mat$ is invariant in law under the action of all
translations of $\RR^d$.  We say that a translation-invariant
matching is a {\bf factor} if it is a deterministic function of
the Poisson process $\poi$ and the mark process $\xi$, that is,
if it does not involve any additional randomness.  We write
$\PP$ and $\E$ for probability and expectation on the
probability space supporting the random triplet
$(\poi,\xi,\mat)$. For later purposes, we note that the notion
of a matching scheme generalizes from the Poisson case to
general simple point processes.

Let $(\poi^*,\xi^*,\mat^*)$ be the Palm version of
$(\poi,\xi,\mat)$ with respect to $\poi$, and write $\PP^*$ and
$\E^*$ for the associated probability law and expectation
operator. Informally speaking, $\PP^*$ describes the
conditional law of $(\poi,\xi,\mat)$ given that there is a
point at the origin, with the mark process and the matching
scheme taken as stationary background; see \cite[Chapter 11]{K}
for more details.  Note that since $\poi$ is a Possion process,
we have $[\poi^*]\stackrel{d}{=} [\poi]\cup \{0\}$. Let $C$
denote the volume of the component of the origin vertex for
$\poi^*$, that is, $C$ is the number of vertices that can be
reached by a path in $G(\poi^*,\mat^*)$ from the origin. Our
first result states that, excluding trivial cases, on one hand
it is always possible to obtain configurations that contain
only finite components in a translation-invariant way, while on
the other hand infinite components can always be achieved.
Furthermore, a connected graph is possible if and only if the
expected degree is at least 2.

\begin{theorem}\label{th:anything}
Let $\poi$ be a Poisson process of intensity $1$ in $\RR^d$,
for any $d\geq 1$, and let $D$ be a random variable with law
$\mu$.
\begin{letlist}
\item For any degree distribution $\mu$, there is a
    simple translation-invariant factor matching scheme
    such that $\PP^*(C<\infty)=1$.
\item If $\PP(D\geq 2)>0$, then there is a simple
    translation-invariant factor matching scheme
    such that $G$ has exactly one infinite component a.s., and
    furthermore $\PP^*(C=\infty\mid D_0\geq 2)=1$.
\item The following are equivalent.\vspace{-2mm}
\begin{romlist}
\item $\E[D]\geq 2$.
\item There exists a simple translation-invariant matching
    scheme for which the graph $G$ is a.s.\ connected.
\item There exists a simple {\em factor} matching scheme
    for which the graph $G$ is a.s.\ connected.
\end{romlist}
\end{letlist}
\end{theorem}

\noindent The implication (ii)$\Rightarrow$(i) of (c) is
analogous to various results for percolation on lattices to the
extent that the expected degree of vertices in infinite
clusters must be at least $2$; see, e.g. \cite[Theorem 2]{G} and
\cite[Theorem 6.1]{BLPS}.

We move on to consider a particular natural type of matching
scheme which in the special case where $\mu(\{1\})=1$ (i.e.,
deterministically one stub per vertex) is known as the {\bf
stable matching}.  See, e.g., \cite{HP}. The natural extension
to general degrees, called the {\bf stable multi-matching}, is
defined as follows; here and throughout, distance $|x-y|$ and
edge length are defined in terms of Euclidean metric on
$\RR^d$.
\begin{defi}\sloppypar
A matching scheme $\mat$ is said to be a {\bf stable
multi-matching} if a.s., for any two distinct points $x,y \in
[\poi]$, either they are linked by an edge or at least one of
$x$ and $y$ has no incident edges longer than $|x-y|$.
\end{defi}
\noindent We remark that the concept of a stable multi-matching
can be defined analogously for general point sets. Here however
we will use the term restricted to the specific situation
described above. We will see in Proposition \ref{prop:sm_works}
in Section \ref{sect:background} below that, for any dimension
$d\geq 1$ and any $\mu$, there is then a unique stable
multi-matching. Our main result on the stable multi-matching is
the following, giving sufficient conditions for existence and
non-existence of an infinite cluster. It may be noted that the
gap between the conditions in (a) and (b) is quite large; see
Section \ref{sect:further} for some discussion on this point.
\begin{theorem}\label{th:sm}
Consider the stable multi-matching.
\begin{letlist}
\item For any $d\geq 2$, there exists a $k=k(d)$ such that
    if $\PP(D\geq k)=1$, then
    $\PP^*(C=\infty)>0$.
\item For any $d\geq 1$ we have that if $\PP(D\leq 2)=1$
    and $\PP(D=1)>0$, then $\PP^*(C=\infty)=0$.
\end{letlist}
\end{theorem}

\noindent
The rest of this paper is organized as follows. In
Section \ref{sect:background} we offer some further background on the
model considered here. In Section \ref{sect:anything} we prove
Theorem \ref{th:anything}, while in Sections \ref{sect:perc} and
\ref{sect:non-perc} we prove parts (a) and (b), respectively, of
Theorem \ref{th:sm}. Finally, in Section \ref{sect:further} we briefly
mention some open problems and scope for further work.


\begin{figure}
\centering
\subfigure[$\PP(D=1)=1-\PP(D=2)=0.05$.]{\epsfig{file=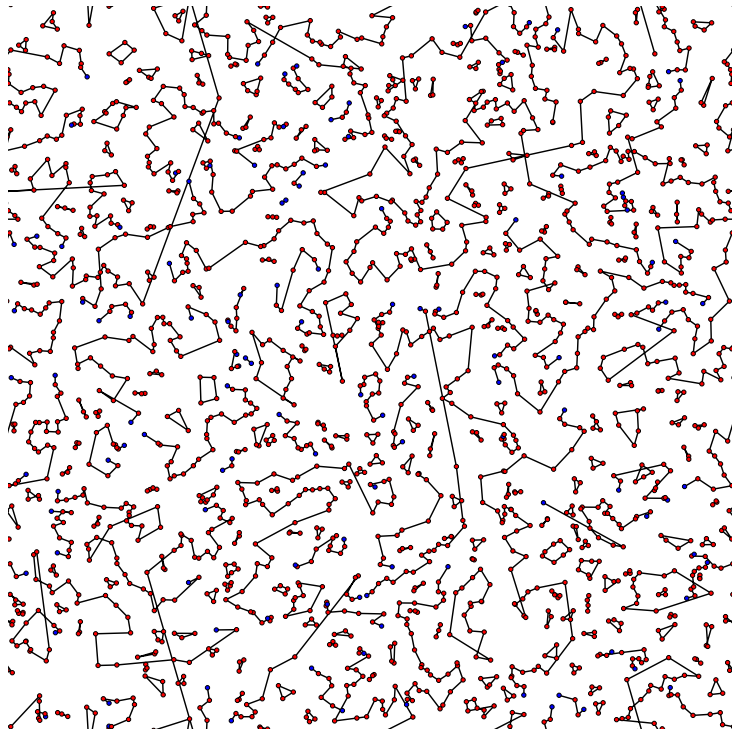,height=9.8cm}}
\subfigure[$\PP(D=2)=1$]{\epsfig{file=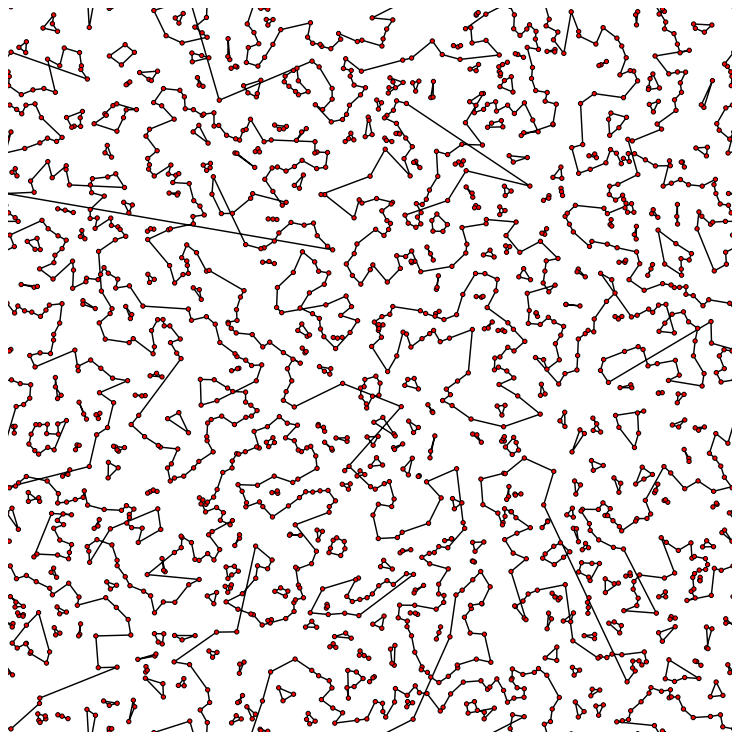,height=9.8cm}}
\vspace{-0.3cm}
\caption{Stable multi-matchings on the torus, with given degree distributions.}
\end{figure}

\begin{figure}
\centering
\subfigure[$\PP(D=3)=1-\PP(D=2)=0.05$.]{\epsfig{file=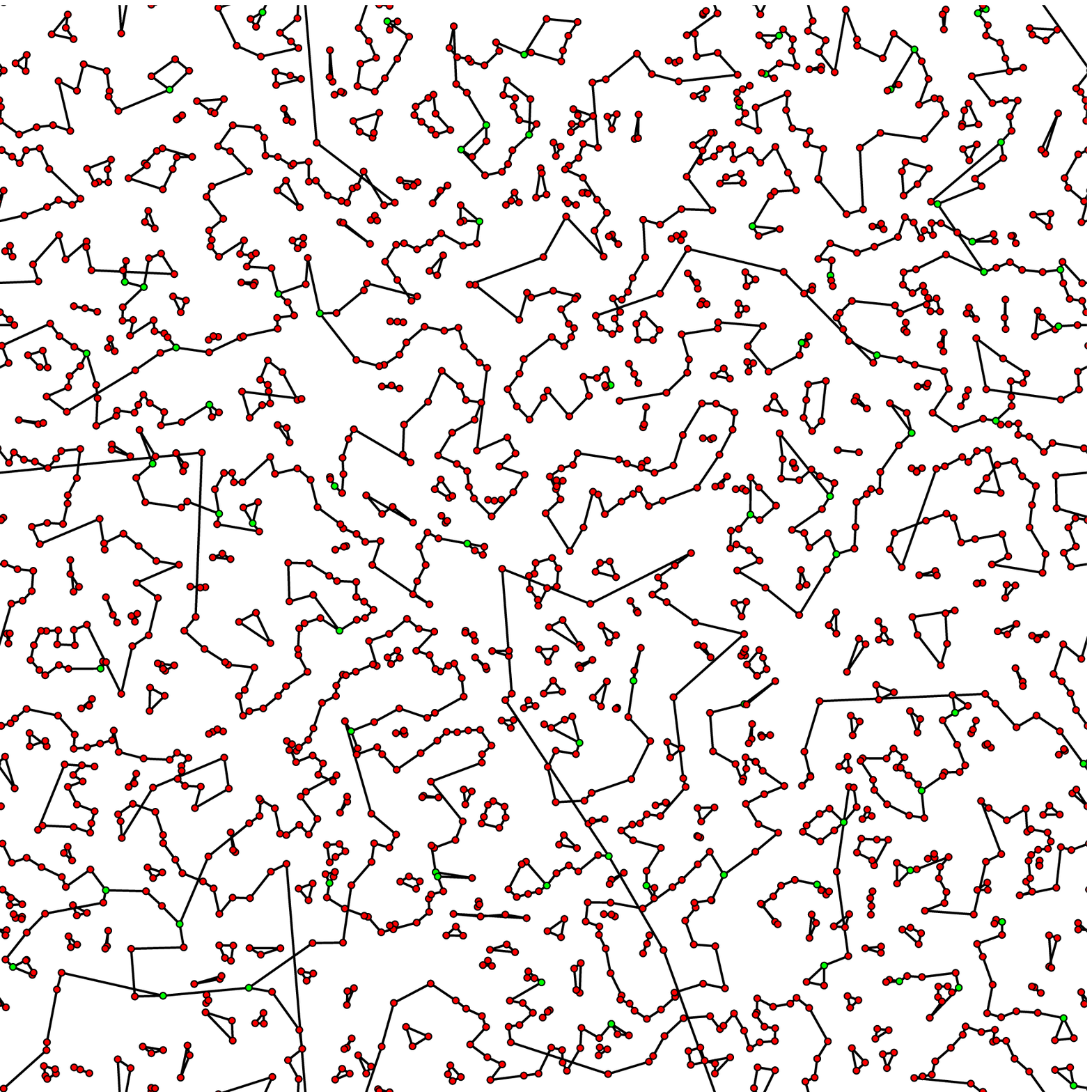,height=9.8cm}}
\subfigure[$\PP(D=3)=1$]{\epsfig{file=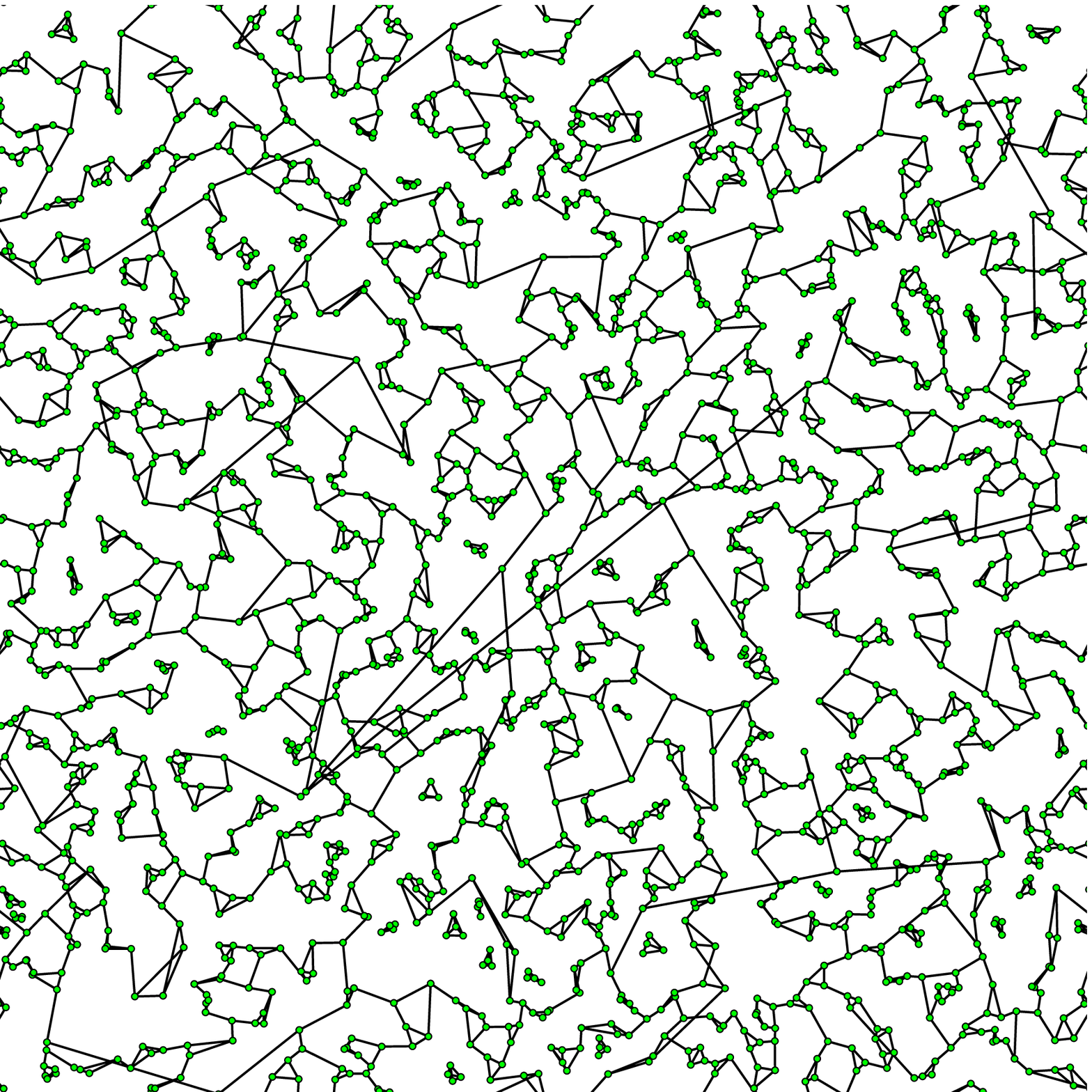,height=9.8cm}}
\vspace{-0.3cm}
\caption{Stable multi-matchings on the torus, with given degree distributions.}
\end{figure}

\section{Background and preliminaries}  \label{sect:background}

\subsection{Random graph models with i.i.d.\  degrees}

Random graphs with prescribed degree distribution have been
extensively studied in non-spatial settings; see e.g.\ \cite{BJR},
\cite{BDM}, \cite{CL1}, \cite{CL2}, \cite{MR1} and \cite{MR2}. Deijfen and Meester
\cite{DM} studied the problem of constructing
translation-invariant graphs with $\Z$ as vertex set and
i.i.d.\  degrees assigned to the vertices. They focussed on the
question of what moment properties on edge lengths are
achievable. Deijfen and Jonasson \cite{DJ} obtained further
results in this direction, which Jonasson \cite{J} extended to
more general deterministic lattices. Finally Deijfen \cite{D}
considered the same problem for the Poisson process $\poi$ on
$\RR^d$, which is the setting we are concerned with here.

\subsection{Stable matchings and stable multi-matchings} \label{sect:background_stable}

The concept of stable matchings goes back to Gale and Shapely
\cite{GS}, and has been extensively studied ever since. Holroyd
and Peres \cite{HP} considered the case of matching Poisson
points in $\RR^d$, while Holroyd et al.\ \cite{HPPS} went on to
consider bipartite matching of two independent Poisson
processes. These last two references provide constructions that
will be useful to us in later sections. We isolate the relevant
findings in the following result concerning the existence of
matchings schemes with constant degree $1$ in
translation-invariant point processes. The {\bf intensity} of a
translation-invariant point process is the expected number of
points in a fixed set of unit volume. A set $U\subset \RR^d$ is
said to be {\bf non-equidistant} if there are no distinct
points $x,y,u,v\in U$ with $|x-y|=|u-v|$ or $|x-y|=|y-v|$,
while a {\bf descending chain} is an infinite sequence
$\{x_i\}\subset U$ such that $|x_i-x_{i-1}|$ is strictly
decreasing.

\begin{prop}[Existence of matchings]\label{prop:usm} $\mbox{ }$
\begin{letlist}
\item Let $\mathcal{R}$ be translation-invariant point
    processes on $\RR^d$ with finite intensity, and assume
    that a.s.\ $[\mathcal{R}]$ is non-equidistant and has
    no descending chains. Then a factor matching scheme for
    $\mathcal{R}$ with constant degree $1$ exists.
\item Let $\mathcal{R}$ and $\mathcal{S}$ be point
    processes on $\RR^d$, jointly ergodic under
    translations, and with equal finite intensities. Assume
    that $[\mathcal{R}]\cup [\mathcal{S}]$ is almost surely
    non-equidistant and has no descending chain. Then there
    exists a factor matching scheme with constant degree
    $1$ for $[\mathcal{R}] \cup [\mathcal{S}]$, having the
    property that every point in $[\mathcal{R}]$ is linked
    to a point in $[\mathcal{S}]$ and vice versa.
\end{letlist}
\end{prop}

\begin{proof}
As an example that proves (a), we can take the stable matching,
whose existence and uniqueness is established in \cite[pp.\ 10-11]{HP}.
For (b) we can take the stable
bipartite matching of $\mathcal{R}$ and $\mathcal{S}$ (i.e.\
the stable matching where two points that are either both in
$\mathcal{R}$ or both in $\mathcal{S}$ are postulated to have
distance $\infty$, while the distance between other pairs of
points is the usual Euclidean one), whose existence and
uniqueness is established in of \cite[Proposition 9]{HPPS}.
\end{proof}

We remark at this point that the Poisson process $\poi$
satisfies the assumptions of Proposition \ref{prop:usm} (a),
because it satisfies the no descending chains property as first
noted in \cite{HM}; see also \cite{DL}.

Moving on to stable multi-matchings, consider the following
procedure for matching the stubs of $(\poi,\xi)$.  In a set of
points $S\subset\RR^d$, call a distinct pair $x,y\in S$ {\bf
mutually closest} if $x$ is the closest point to $y$ in
$S\setminus \{y\}$, and vice-versa.

\begin{list}{Step \arabic{mycount}.\ }{\usecounter{mycount}}
\item Consider the set $[\poi]$ of all points.
    An edge is created between each mutually closest pair
    in this set, and one stub is removed from each of these
    points.

\item Consider the set of all points that still
    have at least one stub after step 1. Two such points
    are called compatible if no edge was created between
    them in step 1. An edge is created between each
    compatible mutually closest pair in this set, and one
    stub is removed from each of these points.
$$
\vdots
$$
\item[Step $n$.] Consider the set of all points that still
    have at least one stub.  Two such points are called
    compatible if no edge has been created between them. An
    edge is created between each compatible mutually
    closest pair in this set, and one stub is removed from
    each of these points.
$$
\vdots
$$
\end{list}

It is immediate that this procedure will not produce self-loops
or multiple edges, and the resulting process is clearly
translation-invariant. We will show that a.s.\ all stubs are
eventually matched, and moreover that the resulting graph is
the unique stable multi-matching of $(\poi,\xi)$.
\begin{prop}\label{prop:sm_works}
Let $(\poi,\xi)$ be a marked Poisson process as before. Almost
surely, the iterative multi-matching procedure described above
exhausts the set of stubs, and the limiting graph (after an
infinite number of iterations) is a stable multi-matching. No
other stable multi-matching of $(\poi,\xi)$ exists.
\end{prop}

\begin{proof} For the case where $\mu(\{1\})=1$ this is an application
the result from \cite{HP} mentioned in the proof
of Proposition \ref{prop:usm} (a). The general case is a
straightforward adaptation of their argument, as follows.

Let $\poi'$ be the process of points with at least one
unmatched stub on them after the above matching procedure is
completed.  Then $\poi'$ is an ergodic point process and hence
has either a.s.\ infinitely many points or a.s.\ no points. To
rule out the former case, call two points in $[\poi']$
compatible if they do not have an edge between them in the
configuration obtained from the matching procedure. Then create
a directed graph $G'$ with $[\poi']$ as vertex set by drawing a
directed edge from each point to its nearest compatible point
(which exists because the initial numbers of stubs were
finite).  Some thought reveals that $G'$ cannot contain cycles
of length more than two, and that each finite component must
contain precisely one cycle of length two. However, a cycle of
length two is also impossible, since it corresponds to two
mutually closest points with no edge between them and an
unmatched stub at each point, and between such points an edge
would indeed have been created at some stage in the matching
procedure.  Hence $G'$ has no finite components, and no cycles.
This implies that if $[\poi']$ is non-empty, then following the
outgoing edges starting at any $x\in [\poi']$ yields a
descending chain. Since descending chains do not exist $\poi$
(and hence not in $\poi'$), we conclude that $[\poi']$ is a.s.\
empty, as desired.

That the resulting multi-matching is stable follows from the
definition: an unstable pair of points would have had an edge
created between them at some stage of the matching procedure.
That it is the unique matching with this property follows by
induction over the stages in the algorithm to show that each
edge that is present in the resulting configuration must be
present in any stable matching of the stubs.
\end{proof}

\begin{remark} \label{rem:sm_works_more_generally}
{\rm Note that the given procedure works and proves Proposition
\ref{prop:sm_works} in the greater generality where the Poisson
process $\poi$ is replaced by any point process satisfying the
requirements of Proposition \ref{prop:usm}.}
\end{remark}

\begin{remark} \label{rem:sm_works_with_some_edges_already_present}
{\rm We will later want to apply the given procedure in
situations where some pairs of vertices already have an edge
between them and additional connections between such vertices
are prohibited.  Provided the existing edges form a
translation-invariant process, the proof of Proposition
\ref{prop:sm_works} shows that the process still exhausts all
remaining stubs, and results in a translation-invariant
process.}
\end{remark}

\subsection{Mass transport}

The so-called mass transport method in percolation theory was
originally developed for the setting of nonamenable lattice
(see \cite{BLPS} for background) but has turned
out to be a convenient tool also for the more familiar setting
of processes living on $\Z^d$ or $\RR^d$. Here we formulate a
special case adapted to the particular needs of the present
paper.  We define a {\bf mass transport} to be a random measure
$T$ on $(\RR^d)^2$ that is invariant in law under translations
of $\RR^d$, that is, $T(A+x,B+x)\stackrel{d}{=}T(A,B)$ for all
Borel $A,B\subseteq\RR^d$ and $x\in\RR^d$, where we write
$T(A,B):=T(A\times B)$.  We interpret $T(A,B)$ as the amount of
mass transported from $A$ to $B$.  Let $Q$ denote the unit cube
$[0,1)^d$.

\begin{lemma}[Mass Transport Principle]\label{le:mt}
Let $T$ be a mass transport. Then
\[
\E\, T(Q,\RR^d) =\E\, T(\RR^d,Q) \, .
\]
\end{lemma}

\begin{proof}
\[
\E\, T(Q,\RR^d)  =  \sum_{z\in\Z^d}\E\, T(Q,Q+z)
 =   \sum_{z\in\Z^d}\E\, T(Q-z,Q) = \E\, T(\RR^d,Q) \, .
\qedhere
\]
\end{proof}

\section{Anything is possible} \label{sect:anything}

The task in this section is to prove Theorem \ref{th:anything},
and we begin with part (a).

\begin{proof}[Proof of Theorem \ref{th:anything} (a)]
 We need to describe a factor matching scheme that gives only finite
components. To this end, let $\poi_n$ denote the process of
points $x\in[\poi]$ with $D_x=n$ (recall that $D_x$ denotes the
number of stubs attached to $x$).  We will partition $[\poi_n]$
into groups of size $n+1$.  The configuration is then taken to
consist of complete graphs on each of these groups.

Take $n$ such that $\poi_n$ is non-empty. To partition
$[\poi_n]$, we assign each point in $[\poi_n]$ a type
$i\in\{1,\ldots,n+1\}$ as follows. Let $R^*$ be the distance
from the origin to the closest other point in the Palm version
of $\poi_n$ and let $0=r_0,r_1,\ldots,r_n,r_{n+1}=\infty$ be
such that
\begin{equation}\label{eq:type_assignment}
\PP^*(r_{i-1}<R^*\leq r_i)=\frac{1}{n+1},\quad i=1,\ldots,n+1.
\end{equation}
For $x\in[\poi_n]$, let $R_x$ denote the distance to the
nearest other point in $[\poi_n]$. We assign $x\in[\poi_n]$
type $i$ if $r_{i-1}<R_x\leq r_i$, and let $\poi_n^i$ be the
process of points of $\poi_n$ of type $i$. Note that this
assignment involves no randomness beyond the Poisson process
itself,
and that for each given $n$, the processes
$\poi_n^1,\ldots,\poi_n^{n+1}$ have equal intensities and are
jointly ergodic under translations. By Proposition
\ref{prop:usm} (b), this means that for each $i=1,\ldots,n$ we
can find a matching scheme that matches each type $i$ point to
a unique type $i+1$ point and vice versa.  The components of
the graph obtained by taking the union of these matchings
partition $[\poi_n]$ into groups of size $n+1$ with one point
of each type in each group.
\end{proof}

For the proofs of parts (b) and (c) of Theorem
\ref{th:anything} the following lemma will be useful.

\begin{lemma} \label{lem:doubly_infinite_paths}
For a Poisson process with exactly $2$ stubs on each point,
there exists a factor matching scheme in which $G$ has a single
component consisting of a doubly infinite path.
\end{lemma}

\begin{proof} This is contained in the proof of \cite[Theorem 1]{HP}.
For expository purposes, let us nevertheless say a few words
about how it is proved. The main step is to construct, in a
translation-invariant way, a one-ended tree whose vertex set is
$[\poi]$. Once that is done, the single doubly infinite path is
easily constructed from the tree by first ordering the children
of each vertex according to distance from the parent, then
ordering all vertices according to depth-first search, and
finally linking each pair vertices that fall next to each other
in this ordering by an edge.
\end{proof}

\begin{remark} \label{rem:weaker_doubly_infinite_paths}
{\rm If we relax the requirement in Lemma
\ref{lem:doubly_infinite_paths} to ask for a union of doubly
infinite paths rather than a single infinite path (this will be
enough for our proof of Theorem \ref{th:anything} (b) but not
for the proof of Theorem \ref{th:anything} (c)), then the tree
construction of \cite{HP} can be replaced by the
following construction: Define the cone $V=\{y\in\RR^d: y_1\geq
|(y_2,\ldots,y_d)|\}$, where $y=(y_1,\ldots,y_d)$, and, for
$x\in[\poi]$, put a directed edge to the (almost surely unique)
point in $(x+V)\cap [\poi]$ whose first coordinate is minimal
among all points in $(x+V)\cap [\poi]$. The resulting graph is
clearly a forest and it is shown in \cite[pp.\ 10-11]{HP} that
the trees are indeed one-ended. Directed infinite
paths can then be created from each tree as in the proof of
Lemma \ref{lem:doubly_infinite_paths}.}
\end{remark}

\begin{proof}[Proof of Theorem \ref{th:anything} (b)]
It is sufficient to provide a factor matching scheme where all
vertices of degree at least 2 belong to a single
infinite component.

To match the stubs in the Poisson configuration in such a way
that an infinite component is obtained we proceed as follows.
First consider all vertices of degree at least 2 and create in
a translation-invariant and deterministic way a directed infinite
path that contains all of them; this is possible by Lemma \ref
{lem:doubly_infinite_paths} (or, if we opt for a union of infinite
paths which is sufficient for the existence but not for the uniqueness
of the infinite component, by the more elementary result in
Remark \ref{rem:weaker_doubly_infinite_paths}). When this is
done we are left with a reduced stub configuration. This is
then matched up using the stable multi-matching described prior
to Proposition \ref{prop:sm_works} with the obvious
modification that we do not allow connections between points
that already have an edge between them arising from the
connections along the paths. Proposition
\ref{prop:sm_works} along with Remark \ref{rem:sm_works_with_some_edges_already_present}
guarantee that this indeed leads to a multi-matching.
\end{proof}

For the proof of Theorem \ref{th:anything} (c), one more lemma
-- a generalization of Proposition \ref{prop:usm} (b) -- will
be convenient.

\begin{lemma} \label{connect_simple_with_multi}
Let $\nu$ be a probability measure on the strictly positive
integers and let $X$ be a random variable with law $\nu$. Let
$\mathcal{R}$ and $\mathcal{S}$ be translation-invariant point
processes on $\RR^d$, jointly ergodic under translations, and
with finite intensities $\lambda_\mathcal{R}$ and
$\lambda_\mathcal{S}$ satisfying
\begin{equation}\label{eq:int_ineq}
\lambda_\mathcal{R} \leq \E[X]\lambda_\mathcal{S} \, .
\end{equation}
Assign degree $1$ to each point in $\mathcal{R}$ and assign
i.i.d.\ degrees with law $\nu$ to the points in $\mathcal{S}$.
If $[\mathcal{R}]\cup [\mathcal{S}]$ is almost surely
non-equidistant and has no descending chains, then there exists
a translation-invariant partial matching scheme, a
deterministic function of $(\mathcal{R},\mathcal{S})$, that
matches each point in $[\mathcal{R}]$ to a stub in
$[\mathcal{S}]$. If (\ref{eq:int_ineq}) holds with equality,
then the procedure also exhausts all stubs in $[\mathcal{S}]$.
\end{lemma}

\begin{proof}
Define
$$
m=\inf\left\{j:\sum_{i=1}^j\PP(X\geq i)\lambda_\mathcal{S}\geq \lambda_\mathcal{R}\right\}
$$
(with $m=\infty$ if (\ref{eq:int_ineq}) holds with equality) and, if $m\geq 2$, let
$$
p_i=\frac{\PP(X\geq i)\lambda_\mathcal{S}}{\lambda_\mathcal{R}}, \quad i=1,\ldots,m-1,
$$
and if $m<\infty$ let $p_m=1-\sum_{i=1}^{m-1}p_i$. If $m=1$,
just let $p_1=1$. As in the proof of Theorem \ref{th:anything}
(a), we let $R^*$ be the distance from the origin to the
closest other point in the Palm version of $\mathcal{R}$ and,
analogously to (\ref{eq:type_assignment}), we define real
numbers $0<r_1<\cdots < r_{m-1}$ such that
\[
\begin{array}{lcl}
\PP^*(R^*\leq r_1) & = & p_1 \\
\PP^*(r_1< R^*\leq r_2) & = & p_2 \\
& \vdots &\\
\PP^*(r_{m-2}<R^*\leq r_{m-1}) & = & p_{m-1}\\
\PP^*(R^*>r_{m-1}) & = & p_m.
\end{array}
\]
For $x \in [\mathcal{R}]$, say that $x$ is of type
$i\in\{1,\ldots,m-1\}$ if $r_i$ is the first number in the
ordered sequence $r_1<\cdots<r_{m-1}$ that exceeds the distance
from $x$ to the nearest other point in $[\mathcal{R}]$, and of
type $m$ if the distance from $x$ to the nearest other point in
$[\mathcal{R}]$ is larger than $r_{m-1}$. This divides the
process $\mathcal{R}$ into processes $\mathcal{R}_i$
($i=1,\ldots,m$) with intensities $\lambda_\mathcal{R}p_i$.
Write $\mathcal{S}_i$ for the process of vertices in
$\mathcal{S}$ with degree at least $i$. By the choice of $p_i$,
for $i=1,\ldots,m-1$, the intensity of $\mathcal{R}_i$
coincides with the intensity of $\mathcal{S}_i$, and condition
\eqref{eq:int_ineq} implies that the intensity of
$\mathcal{R}_m$ is at most by the intensity of $\mathcal{S}_m$
(indeed, the intensity of $\mathcal{R}_m$ is
$\lambda_\mathcal{R}-\sum_{i=1}^{m-1}\PP(X\geq
i)\lambda_\mathcal{S}$ which does not exceed $P(X\geq
m)\lambda_\mathcal{S}$ by the choice of $m$).

Now, for each $i=1,\ldots,m$, match the points of
$\mathcal{R}_i$ to the  points of $\mathcal{S}_i$ using
Proposition \ref{prop:usm} (b). For $i=1,\ldots,m-1$ we get a
perfect matching of the points (since the intensities of the
processes coincide) and, for $i=m$, it is not difficult to see
that all points in $\mathcal{R}_m$ get matched up (while some
points in $\mathcal{S}_m$ may not be used).
\end{proof}

\begin{proof}[Proof of Theorem \ref{th:anything} (c)]
To show that conditions (i), (ii) and (iii) are equivalent, it
suffices to show that
(iii)$\Rightarrow$(ii)$\Rightarrow$(i)$\Rightarrow$(iii). Since
(iii)$\Rightarrow$(ii) is trivial, we only need to show that
(i)$\Rightarrow$(iii) and that (ii)$\Rightarrow$(i).

To prove (i)$\Rightarrow$(iii), first note that the matching scheme
described in the proof of Theorem \ref{th:anything} (b)
gives a connected graph as soon as all vertices have
degree at least 2 (provided we use the construction in
Lemma \ref{lem:doubly_infinite_paths} rather than the one in
Remark \ref{rem:weaker_doubly_infinite_paths}).
To handle the case where $\PP(D=1)>0$, let $X$ be a random variable
distributed as $D-2$ conditional on that $D\geq 3$ and note that
$\E[D]\geq 2$ implies that
\begin{equation}  \label{eq:E=2_rewrite}
\PP(D=1) \leq \E[X]\PP(D\geq 3)
\end{equation}
with equality if and only if $\E[D]=2$.

Consider first the case $\E[D]=2$. For a matching scheme here,
first employ the scheme in the proof of Theorem
\ref{th:anything} (b) in order to connect up all points in
$[\mathcal{R}]$ that are assigned degree $2$ or more into an
infinite path. This leaves the points that are assigned degree
$1$, plus the points initially assigned degree $3$ or more,
each having two of their stubs already matched. Since
(\ref{eq:E=2_rewrite}) holds with equality, Lemma
\ref{connect_simple_with_multi} is exactly tailored to produce
a factor matching of the degree $1$ points to the unmatched
stubs of the degree $\geq3$ points. This gives a connected
graph, so the case $\E[D]=2$ is settled.

For the case $\E[D]>2$ we proceed as with $\E[D]=2$ by first
constructing the infinite path and then connecting up degree
$1$ points to it by the scheme offered in the proof of Lemma
\ref{connect_simple_with_multi}. This time, however, the latter
scheme, although resulting in a connected graph, fails to use
up all the stubs of the degree $\geq 3$ points. These remaining
stubs can be hooked up to each other by the stable multi matching
scheme described prior to Proposition \ref{prop:sm_works} with
the restriction that we do not allow connections between points that
already have an edge between them on the infinite path. Using Remark
\ref{rem:sm_works_with_some_edges_already_present}, this completes the
proof of the (i)$\Rightarrow$(iii) implication.

To prove (ii)$\Rightarrow$(i) we employ a mass-transport
argument. Assume (ii), and let $\mat$ be a matching scheme that
produces a connected graph. Consider the mass transport where
each point $x\in[\poi]$ sends a unit mass to $y\in[\poi]$ if
and only if $x$ and $y$ are connected by an edge, and removing
that edge would leave $x$ in a finite component.  Note that the
mass $M_x^\text{out}$ sent from $x$ cannot exceed 1. We claim
that, for any vertex $x$,
\begin{equation}\label{eq:mass_bound}
D_x-2\geq M_x^\text{in}-M_x^\text{out} \, .
\end{equation}
This follows by considering separately the two cases
$M^\text{out}_x=1$ and $M^\text{out}_x=0$. When
$M_x^\text{out}=1$ we get $M_x^\text{in}= D_x-1$ and
(\ref{eq:mass_bound}) holds with equality. When
$M_x^\text{out}=0$ we have that $x$ is connected to infinity
via at least two edges adjacent to it, which implies that
$M_x^\text{in} \leq D_x-2$, and again (\ref{eq:mass_bound})
holds.

By the mass transport principle (Lemma \ref{le:mt}), the
expectation of the right-hand side of (\ref{eq:mass_bound})
summed over all Poisson points $x$ in the unit cube $Q$ is $0$.
But the expectation of $D_x-2$ summed over all Poisson points
in the unit cube $Q$ is simply $\E[D]-2$ (because the Poisson
process has intensity $1$), so (\ref{eq:mass_bound}) implies
$\E[D]\geq 2$, as desired.
\end{proof}

\section{Percolation in the stable multi-matching} \label{sect:perc}

In this section we prove Theorem \ref{th:sm} (a), that is, we
show that if each Poisson point has sufficiently many stubs
attached to it, then the edge configuration resulting from the
stable multi-matching percolates. The proof uses the result
from \cite{LSS} concerning domination of {\bf
$r$-dependent} random fields by product measures, where a
random field $\{X_z\}_{z\in \Z^d}$ is said to be $r$-dependent
if for any two sets $A,B \in \Z^d$ at $l_\infty$-distance at
least $r$ from each other we have that $\{X_z\}_{z\in A}$ is
independent of $\{X_z\}_{z\in B}$. The version we need is as
follows.

\begin{theorem}[Liggett, Schonmann $\&$ Stacey (1997)]\label{th:lss}
For each $d \geq 2$ and $r\geq 1$ there exists a
$p_c=p_c(d,r)<1$ such that the following holds. For any $r$-dependent
random field $(X_z)_{z\in
\Z^d}$ satisfying $\PP(X_z=1)=1-\PP(X_z=0)\geq p$ with $p
>p_c$, the 1's in $(X_z)_{z\in \Z^d}$ percolate almost
surely.

\end{theorem}

\begin{proof}[Proof of Theorem \ref{th:sm} (a)] The idea of the proof
is a renormalization procedure. We partition $\RR^d$ into cubes
and declare a cube to be \textbf{good} if it contains at least
one but not too many Poisson points and if the same holds for
all cubes close to it, where ``too many'' and ``close'' will be
specified below. We then use Theorem \ref{th:lss} to deduce
that the good cubes can be made to percolate, and we observe
that, if each Poisson point has sufficiently many stubs
attached to it, then each point in a good cube must be
connected to each point in its adjacent cubes in the stable
multi-matching. This forces the existence of an infinite
component in the stable multi-matching.

To make this reasoning more precise,
for $a\in\RR$, let $a\Z^d=\{az: z\in\Z^d\}$
and partition $\RR^d$ into cubes $\{C_{az}\}_{z\in\Z^d}$ centered at
the points of $a\Z^d$ and with side $a$. Two cubes $C_{az}$ and $C_{ay}$
are called adjacent if $|z-y|=1$, and we write $m=m(d)$ for the smallest
integer such that the maximal possible Euclidean distance between points
in adjacent cubes does not exceed $ma$. For each cube $C_{az}$ a super-cube
$S_{az}$ is defined, consisting of the cube itself along with all cubes
$C_{ay}$ with $y$ at $l_\infty$-distance
at most $2m$ from $z$. A super-cube hence
contains $(4m+1)^d$ cubes.

Now, call a cube $C_{az}$ \textbf{acceptable} if it contains at
least one and at most $n=n(d)$ Poisson points, where $n$ will
be specified below, and it is \textbf{good} if all cubes in
$S_{az}$ are acceptable. We have that
$$
\PP(C_{az}\text{ is acceptable})=
1-\PP[\poi(C_{az})=0]-\PP[\poi(C_{az})>n].
$$
The first probability on the right side can be made arbitrarily
small by taking $a$ large, and, for a fixed $a$, the second
probability can be made arbitrarily small by taking $n$ large.
Hence, by choosing first $a$ large and then $n$ very large, we
can make the probability that a cube is good come arbitrarily
close to 1. In particular, by Theorem \ref{th:lss}, we can make
it large enough to guarantee that the good cubes percolate. Fix
such values of $a$ and $n$, let $k=n(4m+1)^d$ and assume that
$\PP(D\geq k)=1$. We will show that then each Poisson point in
a good cube is connected to all Poisson points in the adjacent
cubes, which completes the proof.

Say that a point $x\in[\poi]$ with $D_x$ stubs \textbf{desires}
a point $y\in[\poi]$ if $y$ is one of the $D_x$ nearest points
to $x$ in $[\poi]$. Then a Poisson point $x$ in a good cube
$C_{az}$ desires all points in the adjacent cubes: For any
point $y$ in an adjacent cube, the distance between $x$ and $y$
is at most $ma$, and the Euclidean ball $B_{ma}(x)$ with radius
$ma$ centered at $x$ is contained in the supercube $S_{az}$,
which contains at most $k$ Poisson points (indeed, all
$(4m+3)^d$ cubes in $S_{az}$ are acceptable, which means that
each one of them contains at most $n$ points). Since $D_x\geq
k$, it follows that $y$ desires all points in $B_{ma}(x)$, so
in particular $x$ desires $y$. Furthermore, each Poisson point
$y$ in a cube that is adjacent to a good cube $C_{az}$ desires
each point in the good cube: Since the distance between $x$ and
$y$ is at most $ma$, we have that $B_{ma}(y)\subset
B_{2ma}(x)$. Moreover, the ball $B_{2ma}(x)$ is contained in
the supercube $S_{az}$, which contains at most $k$ points.
Hence $B_{ma}(y)$ contains at most $k$ points and, since
$D_y\geq k$, it follows that $y$ desires all points in
$B_{ma}(y)$.

We have shown that each point in a good cube desires each point
in the adjacent cubes and vice versa. All that remains is to
note that two points that desire each other will indeed be
matched.  This follows from the definition of the stable
multi-matching.
Hence each point in a good cube is connected to each point in
the adjacent cubes and since the good cubes percolate this
proves part (a) of Theorem \ref{th:sm}.
\end{proof}

\begin{remark}
{\rm An easy modification of the above proof reveals that the
following slightly stronger variant of Theorem \ref{th:sm}
holds. For any $\varepsilon>0$, there exists
$k=k(d,\varepsilon)$ such that if $\PP(D\geq k)>\varepsilon$,
then $\PP^*(C=\infty)>0$. }
\end{remark}

\section{Non-percolation in the stable multi-matching}  \label{sect:non-perc}

In this section we prove Theorem \ref{th:sm} (b), that is, we show that
if all points have degree at most 2 and there is a positive probability
for degree 1, then almost surely the stable multi-matching gives
configurations with only finite components. The proof is based on the
following lemma.

\begin{lemma}\label{le:no_si}
In any translation-invariant matching scheme, a.s.\ $G$ has no
component consisting of a singly infinite path.
\end{lemma}

\begin{proof} Assume that components consisting of a singly infinite
paths occur with positive probability, and consider the mass
transport where each Poisson point that sits on a singly
infinite path sends mass 1 to the endpoint of the path. With
positive probability the unit cube $Q$ contains such an
endpoint, and hence the expected mass that is received by $Q$
is infinite. But the expected mass that is sent out from $Q$ is
at most $1$, because the expected number of Poisson points in
$Q$ is $1$. This contradicts the mass transport principle
(Lemma \ref{le:mt}).
\end{proof}

\begin{proof}[Proof of Theorem \ref{th:sm} (b)] Let the degree
distribution be such that $\PP(D\leq 2)=1$ and $\PP(D=1)>0$.
The only infinite components that can occur when $\PP(D\leq
2)=1$ are infinite paths of degree 2 vertices that are
connected to each other and, by Lemma \ref{le:no_si}, any such
path has to be bi-infinite. Assume for contradiction that such
a bi-infinite path occurs with positive probability. We will
describe a coupled configuration of vertex degrees, where, with
positive probability, the edge configuration
 remains unchanged except that a doubly infinite path is cut apart and
turned into two singly infinite paths. This conflicts with
Lemma \ref{le:no_si}.

Given the Poisson process $\poi$ with associated degrees
$\{D_x\}_{x \in [\poi]}$, we now introduce a modified degree
process $\{\widetilde{D}_x\}_{x \in [\poi]}$. Conditional on
$\poi$ and $\{D_x\}_{x \in [\poi]}$, let
$\{\widetilde{D}_x\}_{x \in [\poi]}$ be independent random
variables chosen as follows.  With probability $1-e^{-|x|}$, we
set $\widetilde{D}_x=D_x$, where $|x|$ is the Euclidean
distance from $x$ to the origin. With the remaining probability
$e^{-|x|}$, the degree $\widetilde{D}_x$ is taken to be an
independent random variable with law $\mu$. The Poisson points
that receive a newly generated degree in $\widetilde{D}_x$ are
referred to as \textbf{re-randomized}. Note, crucially, that
$\{\widetilde{D}_x\}_{x \in [\poi]}$ has the same distribution
as $\{D_x\}_{x \in [\poi]}$. Note also that
$$
\E\sum_{x\in[\poi]}\mathbf{1}[x\text{ re-randomized}]=
\int_{\RR^d}e^{-|x|}dx<\infty \, .
$$
Hence, by the Borel-Cantelli lemma, the number
of re-randomized points is finite almost surely.

Now, take a configuration of Poisson points and associated
degrees for which the graph $G$ resulting from the stable
multi-matching contains some bi-infinite path.  Let
$\widetilde{G}$ be the graph resulting from the modified
degrees $\{\widetilde{D}_x\}_{x \in [\poi]}$.  Along each such
path $(x_i)_{i=-\infty}^\infty$, there must be an edge that is
locally maximal, that is, an edge $(x_i,x_{i+1})$ with
$|x_{i+1}-x_i|>\max\{|x_i-x_{i-1}|,|x_{i+2}-x_{i+1}|\}$. To see
this, note that if such an edge did not exist, the vertices of
the path would either constitute a descending chain, or contain
a single locally minimal edge (defined in obvious analogy with
locally maximal). Descending chains do not occur in Poisson
processes (as noted in Section \ref{sect:background_stable})
while chains with a single minimal edge are ruled out by a mass
transport argument similar to the one on the proof of Lemma
\ref{le:no_si} (let each vertex on such a path send unit mass
to the midpoint of the unique invariant edge). Let
$(x_m,x_{m+1})$ be a locally maximal edge -- say the one with a
vertex closest to the origin. Write $A$ for the event that
$x_m$ and $x_{m+1}$ are the only two vertices that are
re-randomized and that
$\widetilde{D}_{x_{m}}=\widetilde{D}_{x_{m+1}}=1$, that is, the
degrees of $x_m$ and $x_{m+1}$ are changed to 1's while the
rest of the degrees remain unchanged.

We claim that, on $A$, the modified graph $\widetilde{G}$
consists of the same edges as in the original $G$ except that
the edge between $(x_m,x_{m+1})$ is absent. Indeed, since
$|x_{m+1}-x_m|\geq \max\{|x_{m+2}-x_{m+1}|,|x_m-x_{m-1}|\}$,
the edge $(x_m,x_{m+1})$ is created at a later stage in the
matching procedure than the edges $(x_{m-1},x_m)$ and
$(x_{m+1},x_{m+2})$. On the event $A$, no stubs have been added
or removed in the modified configuration except that one stub
is taken away from each of $x_m$ and $x_{m+1}$. Hence, up until
the stage when the edge $(x_m,x_{m+1})$ was created in the
original process, precisely the same edges are created in the
modified configuration.  At this stage, the vertices $x_m$ and
$x_{m+1}$ do not have a stub on them, and so the edge
$(x_m,x_{m+1})$ is not created in the modified configuration.
After this stage, the situation is as in the original
configuration, and so the same edges are again created.

The above shows that, on $A$, the modified stable
multi-matching for the modified configuration contains two
singly infinite paths $(x_{m+1},x_{m+2},\ldots)$ and
$(x_m,x_{m-1},\ldots)$. All that remains it to note that, since
the number of re-randomized vertices is finite almost surely,
the event $A$ has positive probability.  We have hence derived
a contradiction with Lemma \ref{le:no_si}.
\end{proof}

\section{Open problems}  \label{sect:further}

\paragraph{Closing the gap in Theorem \ref{th:sm}.} Theorem
\ref{th:sm} provides conditions for when the stable
multi-matching contains an infinite component and for when it
consists only of finite components. These conditions however
are quite far apart and it would be desirable to obtain a more
precise understanding for when the stable multi-matching
percolates. Consider for instance the case with exactly two
stubs attached to each point, that is, $\mu(\{2\})$. Do
infinite components occur in this case? Simulations appear to
suggest a positive answer $d=1$, but are less conclusive in
$d=2$.

Theorem \ref{th:sm} (b) states that percolation does not occur
when there are only degree 1 and degree 2 vertices. Roughly
speaking, this is because the degree 1 vertices serve as dead
ends in the configuration. Does this phenomenon persist when a
small proportion of degree 3 vertices is added? Does a
sufficiently large proportion of degree 1 vertices always
guarantee non-percolation?

Finally, if degree distribution $\mu$ results in an infinite
cluster, and we replace $\mu$ by a distribution $\mu'$ that
stochastically dominates $\mu$, do we still get an infinite
cluster?

\section*{Acknowledgement} Much of this work was carried out at
the 2009 programme in Discrete Probability at the Institut
Mittag-Leffler.  We thank the institute for the generous
support and hospitality. The research of the first author was
supported by the Swedish Research Council. The research of the
second author was supported by the G\"oran Gustafsson Foundation
for Research in the Natural Sciences and Medicine.


\begin{thebibliography}{1}

\bibitem{BLPS} Benjamini, I., Lyons, R., Peres, Y. and Schramm, O.
(1999):  Group-invariant percolation on graphs, \emph{Geom.
Funct. Anal.} \textbf{9}, 29-66.

\bibitem{BJR} Bollob\'{a}s, B., Janson, S. and Riordan, O. (2006):
The phase transition in inhomogeneous random graphs,
\emph{Rand. Struct. Alg.} \textbf{31}, 3-122.

\bibitem{BDM} Britton, T., Deijfen, M. and Martin-L\"{o}f, A.
(2005), Generating simple random graphs with prescribed degree
distribution, \emph{J. Stat. Phys.} \textbf{124},
1377-1397.

\bibitem{CL1} Chung, F. and Lu, L. (2002:1): Connected components
in random graphs with given degrees sequences, \emph{Ann.
Comb.} \textbf{6}, 125-145.

\bibitem{CL2} Chung, F. and Lu, L. (2002:2): The average distances
in random graphs with given expected degrees, \emph{Proc. Natl.
Acad. Sci.} \textbf{99}, 15879-15882.

\bibitem{DL} Daley, D. and Last, G. (2005):
Descending chains, the lilypond model, and mutual nearest neighbour matching,
\emph{Adv. Appl. Prob.} \textbf{37}, 604-628.

\bibitem{D} Deijfen, M. (2009): Stationary random graphs with
prescribed i.i.d.\  degrees on a spatial Poisson process,
\emph{Electr. Comm. Probab.} \textbf{14}, 81-89.

\bibitem{DJ} Deijfen, M. and Jonasson, J. (2006): Stationary
random graphs on $\Z$ with prescribed i.i.d.\  degrees and
finite mean connections, \emph{Electr. Comm. Probab.}
\textbf{11}, 336-346.

\bibitem{DM} Deijfen, M. and Meester, R. (2006): Generating
stationary random graphs on $\mathbb{Z}$ with prescribed
i.i.d.\  degrees, \emph{Adv. Appl. Probab.} \textbf{38},
287-298.

\bibitem{GS} Gale, D. and Shapely, L. (1962): College admissions
and stability of marriage, \emph{Amer. Math. Monthly}
\textbf{69}, 9-15.

\bibitem{G} Gilbert, L. (1995): On the cost of generating an equivalence relation,
\emph{Erg. Theory Dynam. Systems} \textbf{15}, 1173-1181.

\bibitem{HM} H\"{a}ggstr\"{o}m, O. and Meester, R. (1996): Nearest
neighbor and hard sphere models in continuum percolation,
\emph{Rand. Struct. Alg.} \textbf{9}, 295-315.

\bibitem{HPPS} Holroyd, A., Pemantle, R., Peres, Y. and Schramm, O.
(2008): Poisson matching, \emph{Ann. Inst. Henri Poincare}, to
appear.

\bibitem{HP} Holroyd, A. and Peres, Y. (2003): Trees and matchings
from point processes, \emph{Elect. Comm. Probab.} \textbf{8},
17-27.

\bibitem{J} Jonasson, J. (2009): Invariant random graphs with
i.i.d.\  degrees in a general geographgy, \emph{Probab. Th.
Rel. Fields} \textbf{143}, 643-656.

\bibitem{K} Kallenberg, O. (1997): \emph{Foundations of Modern
Probability}, Springer.

\bibitem{L} Levitt, G. (1995): On the cost of generating an
equivalence relation, \emph{Ergodic Theory Dynam. Systems}
\textbf{6}, 1173--1181.

\bibitem{LSS} Liggett, T.,  Schonmann, R. and Stacey M. (1997):
Domination by product measures, \emph{Ann. Probab.}
\textbf{25}, 71-95.

\bibitem{MR1} Molloy, M. and Reed, B. (1995): A critical point for
random graphs with a given degree sequence, \emph{Rand. Struct.
Alg.} \textbf{6}, 161-179.

\bibitem{MR2} Molloy, M. and Reed, B. (1998): The size of the giant
component of a random graphs with a given degree sequence,
\emph{Comb. Probab. Comput.}\ \textbf{7}, 295-305.

\end{thebibliography}
\end{document}